\documentclass[a4paper, 12pt]{amsart}

  \pdfoutput=1
\usepackage{latexsym, amsmath, amssymb, amsthm, mathrsfs, bm,
  mathtools,comment,stmaryrd} \usepackage[mathscr]{eucal}
\usepackage{array} \usepackage{tikz}
\usetikzlibrary{arrows,shapes,positioning,calc,patterns,cd}
\usepackage{graphicx} \usepackage[T1]{fontenc} \usepackage{mathptmx}
\usepackage{microtype} 
\usepackage[colorlinks=true,linkcolor=blue,urlcolor=blue,
citecolor=blue]%
{hyperref}
  
\usepackage[inner=2.3cm,outer=2.3cm, bottom=3.3cm]{geometry}
\usepackage{bbold} \usepackage[initials]{amsrefs}
\usepackage[all,cmtip]{xy} \newtheorem{theorem}{Theorem}[section]
\newtheorem{theoremx}{Theorem}

\newtheorem{lemma}[theorem]{Lemma}
\newtheorem{corollary}[theorem]{Corollary}
\newtheorem{proposition}[theorem]{Proposition}
\theoremstyle{definition} \newtheorem{definition}[theorem]{Definition}

\theoremstyle{remark} \newtheorem{remark}[theorem]{Remark}

\DeclareMathOperator{\rate}{rate} 
\DeclareMathOperator{\edim}{edim}
\DeclareMathOperator{\reg}{reg}

\def\m{\mathfrak{m}} \def\a{\mathfrak{a}} \def\ZZ{\mathbb{Z}}
\def\FF{\mathbb{F}} \def\NN{\mathbb{N}}

\def\Tor{\mathop{\kern0pt\fam0Tor}\nolimits}

\newcommand{\PP}{\mathbb{P}} \newcommand{\ps}[1]{\llbracket {#1}
  \rrbracket} \newcommand{\soc}{\operatorname{soc}}
\DeclareMathOperator{\HF}{HF} \newcommand{\xx}{{\bf x}}
\newcommand{\yy}{{\bf y}} \renewcommand{\aa}{{\bf a}}
\newcommand{\bb}{{\bf b}} \newcommand{\zz}{{\bf z}}
\newcommand{\cc}{{\bf c}} \DeclareMathOperator{\socdeg}{socdeg}

\title{On the rate of generic Gorenstein $K$-algebras}
\author[M. Boij]{Mats Boij} \address{Department of Mathematics, KTH -
  Royal Institute of Technology, SE-100 44 Stockholm, Sweden}
\email{boij@kth.se} \author[E. De Negri]{Emanuela De Negri}
\address{Dipartimento di Matematica, Universit{\`a} di Genova, Via
  Dodecaneso 35, 16146 Genova, Italy} \email{emanuela.denegri@unige.it}
\author[A. De Stefani]{Alessandro De Stefani}
\address{Dipartimento di Matematica, Universit{\`a} di Genova, Via
  Dodecaneso 35, 16146 Genova, Italy} \email{alessandro.destefani@unige.it}
\author[M.E. Rossi]{Maria Evelina Rossi} \address{Dipartimento di
  Matematica, Universit{\`a} di Genova, Via Dodecaneso 35, 16146
  Genova, Italy} \email{rossim@dima.unige.it}

\begin{document}

\begin{abstract} The rate of a standard graded $K$-algebra $A$ is a
  measure of the growth of the shifts in a minimal free resolution of
  $K$ as an $A$-module. In particular $A$ has rate one if and only if
  it is Koszul. It is known that a generic Artinian Gorenstein algebra
  of embedding dimension $n \geq 3$ and socle degree $s=3$ is Koszul.
  We prove that a generic Artinian Gorenstein algebra with $n\geq 4$
  and $s \ge 3 $ has rate $ \lfloor \frac{s}{2} \rfloor. $ In the
  process we show that such an algebra is generated in degree
  $\lfloor \frac{s}{2} \rfloor +1. $ This gives a partial positive
  answer to a longstanding conjecture stated by the first author on
  the minimal free resolution of a generic Artinian Gorenstein ring of
  odd socle degree.

\end{abstract}

\maketitle

\section{Introduction and preliminary results}

Let $K$ be a field and $A=\bigoplus_{i\in \NN}A_i$ be a finitely
generated $\NN$-graded $K$-algebra. We recall that $A$ is said to be
standard graded if $A_0= K$ and $A$ is generated (as a $K$-algebra) by
elements of degree $1, $ that is, $A=K[A_1]$. We denote by ${ \m} $
the homogeneous maximal ideal $\bigoplus_{i\ge 1}A_i$ of $A$. We let
$\edim(A)=\dim_K(A_1)$ be the embedding dimension of $A$,
$\soc(A) = 0:_A\m$ be its socle, which is a homogeneous ideal of $A$,
and $\socdeg(A) = \sup\{i \in \NN \mid \soc(A)_i \ne 0\}$. Observe
that, if $A$ is Artinian, then
$\socdeg(A) = \max\{i \in \NN \mid A_i\ne 0\}$. Consider a minimal
graded free resolution
$$
\FF: \dots \to F_i \to F_{i-1}\to \dots \to F_0\to K\to 0
$$
of $K=A/{\m}$ as an $A$-module.  It is well-known that $\FF$ plays an
important role in the study of the homological properties of $A$.  For
instance, by the Auslander-Buchsbaum-Serre's Theorem, $F_i = 0$ for
$i\gg 0$ if and only if $A$ is regular, i.e., $A$ is isomorphic to a
polynomial ring over $K$.
 
There are several invariants attached to $\FF$. An important one is
the Backelin rate, which measures the growth of the shifts in $\FF$;
we now recall its definition. For integers $i,j \in \ZZ$ we let
$\beta_{i,j}^A(K) = \dim_K(\Tor_i^A(K,K)_j)$ be the $(i,j)$-th graded
Betti number of $K$ in $A$. For every $i$ we let
$t_i^A(K) = \sup \{j \mid \beta^A_{i,j}(K) \neq 0 \}$ where, by
convention, $t_i^A(K) = -\infty$ whenever $\Tor_i^A(K,K)=0$. We
observe that $\Tor_i^A(K,K)$ is a finitely generated graded $K$-vector
space, therefore the integer $t_i^A(K)$ is indeed well-defined. It
readily follows from the given definitions and the
Auslander-Buchsbaum-Serre's Theorem that, if $A$ is not regular, then
$t_i^A(K) \geq i$ for all $i \geq 0$. In the following we assume $A$
not regular, hence $\FF$ is an infinite resolution. The Backelin rate
of $A$ (see \cite{Backelin}) is defined as
\[
  \rate(A)= \sup_{i \geq 2} \left\{\frac{t_i^A(K) - 1}{i-1}\right\}.
\]
 
It turns out that the rate of any standard graded $K$-algebra is
finite, see \cite{Anick,ABH}. Moreover, it follows from our previous
considerations that $\rate(A) \geq 1$ always holds, and one has that
$\rate(A)=1$ if and only if $A$ is Koszul. In this sense, the rate of
$A$ can be regarded as a measure of how far is $A$ from being Koszul.
Consider a minimal presentation of $A$ as a quotient of a polynomial
ring, i.e.  $ A\cong R/I $ where $R=K[x_1,\dots,x_n]$ is a polynomial
ring and $I$ is an ideal generated by homogeneous elements of degree
$>1$ (or, equivalently, $n=\edim(A)$). Let $m(I)$ be the maximum
degree of a minimal homogeneous generator of $I$. It follows from (the
graded version of) \cite[Thm. 2.3.2]{BH} that $t_2(K)=m(I)$, and
therefore one has
\begin{equation}
  \label{minimalRate}
  \rate(R/I) \ge m(I)-1.
\end{equation}
In \cite{Backelin} Backelin proved that $\rate(R/I)=m(I)-1$ for any
monomial ideal $I$ (see also \cite{ERT}). Another interesting class of
standard graded $K$-algebras with minimal rate is given by generic
toric rings, as proved by Gasharov, Peeva and Welker in
\cite{GPW}. Extending a result of Kempf \cite{Kempf}, Conca, De Negri
and Rossi in \cite{CDR} proved that, under suitable conditions, the
coordinate ring of generic points in $\PP^n$ has minimal rate. It is
an interesting problem to provide further classes of algebras of
minimal rate.  \vskip 2mm Artinian Gorenstein $K$-algebras of fixed
embedding dimension $n$ and socle degree $s$ provide a good setting
for studying this question, as they can be parametrized by points in a
suitable projective space via the correspondence given by Macaulay's
inverse system. In fact they are in bijective correspondence with
polynomials in $n$ variables which are homogeneous of degree
$s$. Artinian Gorenstein $K$-algebras which have maximal length in
this family are called compressed; it is known that, when $A$ is a
generic element of such a family, $A$ is compressed \cite{Iarrobino}.
Conca, Rossi and Valla proved that a generic Artinian Gorenstein
algebra $A$ of socle degree three and $n \ge 3 $ is Koszul, see
\cite{CRV}.  The result stimulated several interesting questions and
recent results on $A=R/I$ being Koszul, provided $I$ is generated by
quadrics (see for instance \cite{MSS1,MSS2}).  Observe that if the
socle degree of $A$ is two, then quadratic always implies Koszul.  The
main goal of this paper is to prove the analogue of Conca, Rossi and
Valla's result for higher socle degrees in terms of the rate. We prove
the following result.

\begin{theoremx} \label{thmA} Let $A$ be a generic Artinian Gorenstein
  graded $K$-algebra of socle degree $s \geq 3$, and assume that
  $ \edim(A) \geq 4$. Then $\rate(A) = \lfloor \frac{s}{2} \rfloor$.
\end{theoremx}

A first important step for proving Theorem \ref{thmA} is to provide an
upper bound on $t_i^A(K)$ for all $i \geq 0$ in terms of $m(I), $ see
Proposition \ref{prop compressed}.  Several interesting bounds on
$t_i^A(K)$ have been proved in the literature, see for instance
\cite{ACI1,ACI2,DHS}. However, the one which is needed for our
purposes is obtained from an explicit computation of the bigraded
Poincar\'e series of $K$, see Proposition \ref{Proposition PS}. The
proof is based on the graded version of an argument due to Rossi and
\c{S}ega, used in \cite{RS} in the local case.  As a consequence of
the above results, we prove that if $A=R/I$ is Gorenstein compressed
with socle degree $s \geq 4$ and if $I$ is generated by homogeneous
polynomials of the same degree, then the rate is minimal, that is
$\lfloor \frac{s}{2} \rfloor.$ Then the key step in the proof of our
main result is to show that a $K$-algebra satisfying the assumptions
of Theorem \ref{thmA} has generators in the same degree.

To this end, recall that a generic Artinian Gorenstein $K$-algebra $A$
of socle degree $s=3$ and embedding dimension $n \geq 3$ is generated
by quadrics. An Artinian compressed Gorenstein algebra of even socle
degree $s=2(t-1)\ge 2$ has almost linear resolution; as a consequence,
it is minimally generated only in degree $t$, see \cite[Example
4.7]{Iarrobino}. On the other hand, an Artinian compressed Gorenstein
$K$-algebra $A=R/I$ of odd socle degree $s=2t-1 \geq 3$ always has
generators in degree $t$, and it has no generators in degrees
different than $t$ or $t+1$ \cite[Proposition 3.2]{B99}. If $A=R/I$ is
a generic Artinian Gorenstein $K$-algebras of odd socle degree
$s=2t-1 \geq 3$ and $n \geq 4$, the graded Betti numbers in a minimal
free resolution of $A$ as $R$-module were conjectured by Boij
\cite[Conjecture 3.13]{B99} in terms of $n$ and $t. $ For every
homological position there exists a Zariski open set (depending on the
coefficients of the form $F$ of degree $s$) where the Artinian
compressed Gorenstein algebras have the same graded Betti numbers.
The conjecture in particular predicts that $I$ should be minimally
generated in degree $t$.

In this article we give a positive answer to this part of the
conjecture. In order to do this, we need to show that the Zariski open
set parametrizing the Artinian compressed Gorenstein algebras
generated in the same degree $t$ is not empty.

\begin{theoremx} \label{thmB} For all integers $t \ge 2$ and $n \ge 4$
  there exists an Artinian compressed Gorenstein $K$-algebra $A= R/I$
  with socle degree $s=2t-1$ and embedding dimension $n$ such that $I$
  is generated in degree $t$.
\end{theoremx}

The first step in our proof is to produce an Artinian level algebra
with socle degree $s$ defined by monomials of degree $s$, see Lemma
\ref{lemma:1}.  The required Artinian compressed Gorenstein
$K$-algebra is then obtained from it using a tricky construction, see
Theorem \ref{generation}.

We point out that Boij's conjecture fails in general for generic
Artinian Gorenstein algebras but in the middle of the resolution, not
at the level of generators, consistently with Theorem \ref{thmB}. In
fact, counterexamples were exhibited by Kunte \cite{Ku}, see Corollary
5.6 and Remark 5.7.

\section{The graded Poincar\'e series of a compressed Gorenstein $K$-algebra}
Let $A$ be a standard graded $K$-algebra and $M$ be a finitely
generated $\ZZ$-graded $A$-module.

The graded Poincar\'e series of $M$ is defined as
\[
  P^A_M(u,v) = \sum_{i \geq 0}\left(\sum_{j \in
      \ZZ}\beta_{i,j}^A(M)v^j\right)u^i \in \ZZ[v,v^{-1}]\ps{u},
\]
where $\beta_{i,j}^A(M) = \dim_K(\Tor_i^A(M,K)_j)$. Note that, fixed
an integer $i \geq 0$, one has that $\Tor_i^A(M,K)_j=0$ for all but
finitely many $j \in \ZZ$. Therefore each sum
$\sum_{j \in \ZZ} \beta_{i,j}^A(M)v^j$ is finite.  \vskip 2mm

Given a minimal presentation $A=R/I$ with $R=K[x_1,\ldots,x_n]$, Serre
\cite{Serre} showed that there is a coefficient-wise inequality
\begin{equation} \label{eqGolod} P^A_K(u,v) \leq
  \frac{P^R_K(u,v)}{1-u(P^R_A(u,v)-1)}.
\end{equation}
\begin{definition} A standard graded $K$-algebra $A$ is said to be
  Golod in the graded sense if equality holds in (\ref{eqGolod}).
\end{definition}
\begin{remark} \label{remark Golod} Since for total Betti numbers one
  has $\beta_i^A(K) = \beta_i^{A_\m}(A_\m/\m A_\m)$ for all $i$, and
  because one always has a point-wise inequality between graded Betti
  numbers in (\ref{eqGolod}), it turns out that $A$ is Golod in the
  graded sense if and only if $A_\m$ is a Golod local ring.
\end{remark}

With the aim of bounding $t_i^A(K)$ for every $i>0$, the goal of this
section is to compute the graded Poincar\'e series of $K = A/\m$ for
any Artinian compressed Gorenstein algebra $A$.

Assume that $A$ is an Artinian Gorenstein standard graded
$K$-algebra. Then we can write $A=\bigoplus_{i=0}^s A_i$ with
$A_s \cong K$. The integer $s$ is called the socle degree of $A$. The
Hilbert function of $A$ is the function
$\HF_A:\NN \to \NN$, defined as $\HF_A(i) = \dim_K(A_i)$. The Hilbert function of a Gorenstein algebra is symmetric, that is, $\HF_A(i) = \HF_A(s-i)$ for all $i$. In particular, since $\HF_A(i) \leq \binom{n+i-1}{i}$, one
has
\[
  \HF_A(i) \le \min\left\{ \binom{n+i-1}{i}, \binom{n+s-i-1}{s-i}
  \right\}.
\]

\begin{definition} An Artinian Gorenstein $K$-algebra $A$ of socle
  degree $s$ is said to be {\it{ compressed}} if for all
  $i \in \{0,\ldots,s\}$ one has
  \[
    \HF_A(i) = \min\left\{ \binom{n+i-1}{i}, \binom{n+s-i-1}{s-i}
    \right\}.
  \]
\end{definition}
\vskip 2mm
 
In what follows, given a $\ZZ$-graded module $M = \bigoplus_{j \in \ZZ}M_j$ and $a \in \ZZ$, for all $j \in \ZZ$ we will set $M(a)_j = M_{a+j}$, and we will let $M(a) = \bigoplus_{j \in \ZZ} M(a)_j$ be the shift of $M$ by $a$. The following result is a graded version of \cite[Proposition
6.2]{RS}.

\begin{proposition} \label{Proposition PS} Let $A$ be an Artinian
  compressed Gorenstein $K$-algebra of embedding dimension $n \ge 2 $
  and socle degree $s$, with $2 \leq s \ne 3$. Then
  \[
    P^A_K(u,v) = \frac{(1+uv)^n}{1-u(P^R_A(u,v) - 1) +
      u^{n+1}v^{n+s}(1+u)}.
  \]
\end{proposition}
\begin{proof}
  By \cite{LevinAvramov}, the natural map $A \to T = A/\soc(A)$ is
  Golod, therefore
  \[
    P^T_K(u,v) = \frac{P^A_K(u,v)}{1-u(P^A_T(u,v)-1)}.
  \]
  From the graded short exact sequence
  $0 \to \soc(A) \cong K(-s) \to A \to T \to 0$, we see that
  $P^A_T(u,v) = 1 + uv^sP^A_K(u,v)$, and substituting above we get
  \[
    P^T_K(u,v) = \frac{P^A_K(u,v)}{1-u^2v^sP^A_K(u,v)},
  \]
  from which we conclude
  \[
    P^A_K(u,v) = \frac{P^T_K(u,v)}{1+u^2v^sP^T_K(u,v)}.
  \]
  Since $A_\m$ satisfies the assumptions of \cite[Theorem 5.1]{RS}, we
  have that $T_\m$ is a Golod ring (see the paragraph before
  \cite[Proposition 6.3]{RS}). Thus, we have that $T$ is Golod in the
  graded sense by Remark \ref{remark Golod}, and therefore
  \[
    P^T_K(u,v) = \frac{P^R_K(u,v)}{1-u(P^R_T(u,v)-1)} =
    \frac{(1+uv)^n}{1-u(P^R_T(u,v)-1)}.
  \]
  Consider again the graded short exact sequence
  $0 \to \soc(A) \cong K(-s) \to A \to T \to 0$. As the map
  $\Tor_i^R(\soc(A),K) \to \Tor_i^R(A,K)$ is zero for $i \ne n$ (see
  the proof of \cite[Lemma 6.1]{RS}), we obtain the following exact
  sequences
  \[
    \xymatrix{ 0 \ar[r] & \Tor_i^R(A,K) \ar[r] & \Tor_i^R(T,K) \ar[r]
      & \Tor_{i-1}^R(K,K)(-s) \ar[r] & 0 }
  \]
  for all $0 < i < n$, and
  \[
    \xymatrix{ 0 \ar[r] & \Tor_n^R(K,K)(-s) \ar[r] & \Tor_n^R(A,K)
      \ar[r] & \Tor_n^R(T,K) \ar[r] & \Tor_{n-1}^R(K,K)(-s) \ar[r] &
      0.  }
  \]
  Therefore for $0<i<n$ we have
  \[
    \beta_{i,j}^R(T) = \beta_{i,j}^R(A) + \beta_{i-1,j-s}^R(K),
  \]
  while for $i=n$ we have
  \[
    \beta_{n,j}^R(T) = \beta_{n,j}^R(A) + \beta_{n-1,j-s}^R(K) -
    \beta_{n,j-s}^R(K).
  \]
  Putting all together we have
  \begin{align*}
    P^R_T(u,v) &= \sum_{i,j} \beta_{i,j}^R(T)u^iv^j \\
               & = \sum_{\substack{0 \leq i \leq n \\ j \in \ZZ}}  \beta_{i,j}^R(A) u^iv^j + uv^s \sum_{\substack{1 \leq i \leq n \\ j \in \ZZ}} \beta_{i-1,j-s}^R(K)u^{i-1}v^{j-s} - \sum_{j \in \ZZ}  \beta_{n,j-s}^R(K) u^nv^j\\
               & = P^R_A(u,v) + uv^sP^R_K(u,v) - u^n \sum_{j \in \ZZ} \left(\beta_{n,j-s}^R(K)u v^j + \beta_{n,j-s}^R(K)v^j\right) \\
               & = P^R_A(u,v) + uv^sP^R_K(u,v)  - u^n(u+1)v^{n+s} = P^R_A(u,v) + uv^s(1+uv)^n  - u^n(u+1)v^{n+s}
  \end{align*}
  where we used that
  $\beta_{n,j-s}^R(K) = \begin{cases} 1 & \text{ if } j-s = n \\ 0 &
    \text{ otherwise.} \end{cases}$

  Substituting the identities obtained so far in the formula for
  $P^A_K(u,v)$, we finally get that
  \begin{align*}
    P^A_K(u,v) & = \frac{P^T_K(u,v)}{1+u^2v^sP^T_K(u,v)} \\
               & = \frac{(1+uv)^n}{1-u(P^R_T(u,v)-1)+u^2v^s(1+uv)^n} \\
               & = \frac{(1+uv)^n}{1-u(P^R_A(u,v)-1) + u^{n+1}v^{n+s}(1+u)}. \qedhere
  \end{align*}
\end{proof}

\begin{proposition} \label{prop compressed} Let $A$ be an Artinian
  compressed Gorenstein graded $K$-algebra of embedding dimension
  $n \geq 2$ and socle degree $s \geq 4$, and let $A=R/I$ be a minimal
  presentation. Then $t_i^A(K) \leq (m(I)-1)(i-1)+1$ for every
  $i\geq 1$.
\end{proposition}
\begin{proof}
  Let $t= \lfloor \frac{s}{2} \rfloor+1$. Since $A$ is a compressed
  Gorenstein $K$-algebra, by \cite{FrobergLaksov} (see also
  \cite{B99}) a graded minimal free resolution of $A$ over $R$ has the
  following shape:
  \[
    \xymatrixrowsep{1mm} \xymatrix{
      &&R(-(n+t-2))^{a_{n-1}}  && R(-t)^{a_1} \\
      0 \ar[r] & R(-(n+s)) \ar[r] & \qquad \bigoplus \qquad \ar[r] & \ldots \ar[r] & \qquad \bigoplus \qquad \ar[r]& R \ar[r] & A \ar[r]&0.  \\
      &&R(-(n+t-1))^{b_{n-1}} &&R(-(t+1))^{b_1} }
  \]
  By Proposition \ref{Proposition PS} we therefore obtain
  \begin{align*}
    P^A_K(u,v) & = \displaystyle \frac{(1+uv)^n}{1-u(\sum_{j=1}^{n-1} (a_j+b_jv)u^jv^{t+j-1}) + u^nv^{n+s}) + u^{n+1}v^{n+s}(1+u)} \\
               & =  \frac{(1+uv)^n}{1-\sum_{j=1}^{n-1} (a_j+b_jv) u^{j+1}v^{t+j-1} + u^{n+2}v^{n+s}}. 
  \end{align*}
  Write $P^A_K(u,v) = \sum_{i \geq 0} \beta_i(v)u^i$, where
  $\beta_i(v) \in \ZZ[v]$. For simplicity, for all $i \in \NN$ we let
  $\delta_i=\deg(\beta_i(v))=t_i^A(K)$, where $\deg(-)$ denotes the
  degree of a polynomial in the variable $v$. By clearing denominators
  and equating the coefficients of $u^i$ in both sides of the previous
  identity we deduce that
  \begin{equation}
    \label{casesBetti}
    \beta_i = \begin{cases} 1 & \text{ if } i=0 \\
      nv & \text{ if } i=1 \\
      {n \choose i}v^i + \sum_{j=1}^{i-1} \beta_{i-j-1}(a_j+b_jv)v^{t+j-1} & \text{ if } 2 \leq i \leq n \\
      \sum_{j=1}^{n-1} \beta_{n-j+1}(a_j+b_jv)v^{t+j-1} + v^{n+s} & \text{ if } i=n+2 \\
      \sum_{j=1}^{n-1} \beta_{i-j-1}(a_j+b_jv)v^{t+j-1} & \text{ otherwise}
    \end{cases}.
  \end{equation} 

  We distinguish two cases. If $b_1=0$, then by induction on
  $i \geq 1$ we prove that $\delta_i \leq (t-1)(i-1)+1$. Note that for
  $i=1,2$ the inequality is satisfied since $\delta_1=1$ and
  $\delta_2=t$; thus we may assume $i \geq 3$. For
  $2 \leq j=i-1 \leq n-1$ we have that
  \[
    \deg(\beta_0(a_{i-1}+b_{i-1}v)v^{t+i-2}) \leq t+i-1 \leq
    (t-1)(i-1)+1.
  \]
  For $1 \leq j \leq n-1$ with $j\ne i-1$, if we assume that
  $\delta_{i-j-1} \leq (t-1)(i-j-2)+1$, then
  \begin{equation} \label{Deg1} \deg(\beta_{i-j-1}(a_j+b_jv)v^{t+j-1})
    \leq (t-1)(i-j-2)+1+t+j \leq (t-1)(i-1)+1,
  \end{equation}
  where the latter holds since $t \geq 2$ in our assumptions. One can
  check directly that $i \leq (t-1)(i-1)+1$ holds. Finally, since
  $s \leq 2t-1$, $n \geq 2$ and $t \geq 3$, we also have that
  $n+s \leq (t-1)(n+1)-1$. All these inequalities, together with
  (\ref{casesBetti}) and the inductive hypothesis which allows to use
  (\ref{Deg1}) for appropriate values of $j$, give that
  $\delta_i \leq (t-1)(i-1)+1$. Note that the assumption $b_1=0$
  guarantees that $m(I)=t$, and since $\delta_i=t_i^A(K)$ we proved
  that $t_i^A(K) \leq (m(I)-1)(i-1)+1$, as desired.

  Now assume that $b_1\ne 0$, so that $m(I)=t+1$. In this case, we
  prove by induction on $i \geq 1$ that $\delta_i \leq t(i-1)+1$. The
  strategy is similar to the one used in the previous case. Since
  $\delta_1=1$ the claimed inequality is true for the base case, and
  we may assume $i \geq 2$. For $1 \leq j=i-1 \leq n-1$ we have that
  \[
    \deg(\beta_0(a_{i-1}+b_{i-1}v)v^{t+i-2}) \leq t+i-1 \leq t(i-1)+1.
  \]
  Observe that, for $1 \leq j \leq n-1$ with $j \ne i-1$, if we assume
  that $\delta_{i-j-1} \leq t(i-j-2)+1$, then
  \begin{equation} \label{Deg2} \deg(\beta_{i-j-1}(a_j+b_jv)v^{t+j-1})
    \leq \delta_{i-j-1}+t+j \leq t(i-j-2)+1+t+j \leq t(i-1)+1.
  \end{equation}
  Finally, we have that $i \leq t(i-1)+1$ for all $i \geq 2$ and
  $n+s \leq t(n+1)-1$ since $s \leq 2t-1$. These inequalities,
  together with (\ref{casesBetti}) and the inductive hypothesis which
  allows to use (\ref{Deg2}) for all appropriate values of $j$, give
  that $t_i^A(K) = \delta_i \leq t(i-1)+1 = (m(I)-1)(i-1)+1$.
\end{proof}

In \cite[Question 6.10]{CRV} the authors ask if any quadratic
Gorenstein algebra of socle degree three is Koszul.  However, for any
embedding dimension $n \geq 8$ there are examples of Gorenstein
quadratic algebras of regularity three which are not Koszul, see
\cite{MSS1,MCS}.

In contrast with the case of socle degree three, compressed Gorenstein
algebras with $s \geq 4$ have minimal rate.

\begin{corollary} \label{corollary t-generated} Let $A$ be an Artinian
  compressed Gorenstein graded $K$-algebra of socle degree $s \geq 4$,
  and let $A=R/I$ be a minimal presentation.  Then
  $\rate(A) = m(I)-1$.
\end{corollary}
 \begin{proof}
The statement follows immediately
from Proposition \ref{prop compressed} and (\ref{minimalRate}).
 \end{proof}

\section{The rate of a generic Gorenstein $K$-algebra}

Every Artinian Gorenstein graded algebra $A=R/I$ of socle degree $s$,
corresponds, up to scalars, to a form $F$ of degree $s$ in another set
of variables as follows. Let $R=K[x_1, \dots, x_n] $ and
$B=K[y_1, \dots, y_n]. $ Regard $B$ as a $R$-module via the action
$x_i \circ F= \partial F/\partial y_i; $ this action can be extended
to every element $G \in R $ acting as a differential operator on the
elements of $B.$ In case of positive characteristic, we have to
replace the differentiation action by the \emph{contraction} action
where
\[
  x_1^{\alpha_1}
  x_2^{\alpha_2}\cdots
  x_n^{\alpha_n}\circ   y_1^{\beta_1}
  y_2^{\beta_2}\cdots
  y_n^{\beta_n}  = \begin{cases}
    y_1^{\beta_1-\alpha_1}
  y_2^{\beta_2-\alpha_2}\cdots
  y_n^{\beta_n-\alpha_n},  & \text{ if $\beta_i\ge \alpha_i$ for
    $i=1,2,\dots,n$}\\
  0, &\text{otherwise.}\\
  \end{cases}
\]
Given a form $F$ of degree $s, $ we denote by $I_F$
the ideal of the elements of $R$ which annihilate $F$:
$$ I_F= \operatorname{ann}(F) = \{G\in R \ | \ G \circ F=0\}$$
and we denote $A_F= R/I_F.$ Then $A_F$ is a standard graded Artinian Gorenstein algebra
of socle degree $s.$ Moreover, every ideal $I \subset R$ which defines
a standard graded Artinian Gorenstein algebra of socle degree $s$
arises in this way. This correspondence is known as Macaulay's Inverse
System. For more information we refer to the book of Iarrobino and
Kanev \cite{IK}. In particular, this means that the set of standard graded Artinian Gorenstein algebras of socle degree $s$ and embedding dimension $n$ is parametrized by
the projective space $\PP(B_s)$. For a property that holds for all
algebras in an open dense subset of this parameter space, we say that
the property holds for a \emph{generic}  Artinian Gorenstein
$K$-algebra $A$ of socle degree $s$. For example, Iarrobino proved
the following theorem
\cite[Prop 4.7]{Iarrobino}.

\begin{theorem}\label{compressed}
  A generic Artinian Gorenstein algebra of socle degree $s$ is
  compressed. 
\end{theorem}

\subsection{Generators of generic Artinian Gorenstein algebras}

It is known that Artinian compressed Gorenstein $K$-algebras of even
socle degree $s=2(t-1)>1 $ have almost linear resolution (see
\cite[Example 4.7]{Iarrobino}). That is, if $n=\edim(A)$ and $A=R/I$ is a
minimal presentation, then a minimal free resolution of $A$ over $R$
is almost linear:
\[
  \xymatrixcolsep{4.5mm} \xymatrix{ 0 \ar[r] & R(-(n+s)) \ar[r] &
    R^{a_{n-1}}(-(n+t-2)) \ar[r] & \ldots \ar[r] & R^{a_2}(-(t+1)) \ar[r] &
    R^{a_1}(-t) \ar[r] & R \ar[r] & A \ar[r] & 0.  }
\]
In particular
\[
  m(I)=t=s/2 + 1.
\]
\vskip 2mm

If $s=2t-1$ and $n=3$ the Betti numbers of generic Gorenstein algebras
are known, see for instance \cite[Example 3.16]{B99}, and the
following holds:
\[
  m(I)  = \begin{cases} t & \text{ if } t \text{ is even}\\
    t +1 & \text{ if } t \text{ is odd.}
  \end{cases}
\]

\begin{remark} \label{Boij Conj}
Boij proved that there exists a non-empty Zariski open set of Artinian compressed Gorenstein algebras
all sharing the same numerical resolution (see \cite[Corollary 3.10]{B99}). In the same paper, he conjectured an explicit formula for the graded Betti numbers of generic Artinian Gorenstein algebras with $n>3$ and socle degree $s>1$ \cite[Conjecture 3.13]{B99}, which predicts that the defining ideal is generated in a single degree. Note that, in order to prove the conjecture, one only has to exhibit an explicit example of a compressed Gorenstein algebra with the expected resolution. If one only produces an example where the defining ideal is generated in a single degree, then there is a non-empty Zariski open set of algebras sharing the same property.
\end{remark}

Regarding Boij's conjecture, computer generated evidence was given in
\cite[Section 4]{B99}, but it was later disproved for
special values of $n$ and $s$ by Kunte \cite[Corollary 5.6,\ Remark
5.7]{Ku}. However, in all known counterexamples the failure is in the
middle of the resolution. The aim of this section is to prove that
Boij's conjecture holds true at the beginning of the
resolution. Namely, we prove that generic Artinian Gorenstein algebras
of socle degree $s=2t-1$ are minimally generated in degree $t$.

We recall that an Artinian $K$-algebra $A$ of socle degree $s$ is
called \emph{level} if $\soc(A) \subseteq A_s$. The construction of the
following lemma will be crucial in what follows.

\begin{lemma}\label{lemma:1}  
  For all integers $t \geq 2$ and $m\geq 3$ there exists a monomial
  ideal $J \subseteq S=K[x_1,\ldots,x_m]$ such that $m(J)=t$ and
  $A=S/J$ is an Artinian level algebra with socle degree $t$.
\end{lemma}

\begin{proof}
  We start by considering the even case $t=2u$ with $u \geq 1$. Given
  $m \geq 3$, we let $j$ be an integer such that $1 \leq j < m$. We
  set $X=(x_1,\ldots,x_j)$, $Y=(x_{j+1},\ldots,x_m)$ and $\m=X+Y$. We
  claim that the ideal $J=(X^2+Y^2)^u$ satisfies the desired
  conditions. It is clear that $J$ is a monomial ideal generated in
  degree $2u=t$. In order to show that $A=S/J$ is level of socle
  degree $t$, we first set up some notation: we let
  $\xx=x_1\cdots x_j$, and given $\aa=(a_1,\ldots,a_j) \in \NN^j$ and
  we set $\xx^{\aa} = x_1^{a_1} \cdots x_j^{a_j}$. Similarly we let
  $\yy=x_{j+1}\cdots x_m$, and given
  $\bb=(b_{j+1},\ldots,b_m) \in \NN^{m-j}$ we set
  $\yy^{\bb} = x_{j+1}^{b_{j+1}} \cdots x_m^{b_{m}}$. Note that $S$ is
  bigraded by setting $\deg(x_i)=(1,0)$ for all $x_i \in X$, and
  $\deg(x_i) = (0,1)$ for all $x_i \in Y$. In this way, if we let
  $a=|\aa| = \sum_{i=1}^j a_i$ and $b=|\bb| = \sum_{i=j+1}^m b_i$,
  then the monomial $\xx^\aa\yy^\bb$ is bigraded of degree
  $(a,b)$. Observe that the generators of $J$ are precisely the monomials
  $\xx^\aa\yy^\bb$ of degree $a+b=t$ and such that both $a$ and $b$
  are even.

  To show that $A$ is level with $\socdeg(A)=t$ we first prove that
  $\m^{t+1} \subseteq J$. Let $\xx^\aa\yy^\bb$ be a monomial of degree
  $a+b \geq t+1 = 2u+1$. It is easy to see that there exist vectors
  $\aa' \leq \aa$ and $\bb' \leq \bb$ (where the inequality is
  intended point-wise) such that $|\aa'| =2h$ and $|\bb'|=2(u-h)$ for
  some $h \in \NN$. It follows that
  $\xx^\aa \in (\xx^{\aa'}) \subseteq X^{2h} = (X^2)^h$ and
  $\yy^\bb \in (\yy^{\bb'}) \subseteq Y^{2(u-h)} =
  (Y^2)^{u-h}$. Therefore, we conclude that
  $\xx^\aa\yy^\bb \in (X^2)^h(Y^2)^{u-h} \subseteq J$. This shows that
  $\socdeg(A) \leq t$. Now we show that $\socdeg(A) \geq t$. Since $J$
  is a monomial ideal, so is $J:\m$. Therefore, by degree
  considerations we only have to show that no monomial of degree
  $t-1=2u-1$ belongs to $J:\m$. But if $\xx^\aa \yy^\bb$ has degree
  $a+b=t-1$, then either $a$ is odd and $b$ is even or vice
  versa. Assume without loss of generality that we are in the first
  case; if by contradiction we assume that $\xx^\aa \yy^\bb \in J:\m$,
  then $\xx^\aa \yy^\bb \cdot x_{j+1} = \xx^\aa\yy^{\bb'} \in J$,
  where $\bb'= \bb + (1,0\ldots,0)$. By degree considerations,
  $\xx^\aa\yy^{\bb'}$ must then be one of the minimal generators of
  $J$; however, this contradicts our previous characterization of such
  generators, since $a$ is odd. This concludes the proof in the case
  $t$ is even.

  We now focus on the case in which $t=2u+1$ is odd, with $u \geq
  1$. Given $m \geq 3$, we let $j,k$ be integers such that
  $1 \leq j < k \leq m$, and we let
  $X=(x_1,\ldots,x_j), Y=(x_{j+1},\ldots,x_k)$, and
  $Z=(x_{k+1},\ldots,x_m)$. We also let $\m=X+Y+Z$. We claim that the
  ideal
  $J=X(X^2+Y^2)^u+Y(Y^2+Z^2)^u+Z(Z^2+X^2)^u+XYZ((X+Y)^2+Z^2)^{u-1}$
  satisfies the desired conditions. It is again clear that $J$ is a
  monomial ideal generated in degree $t$.

  To show that $A$ is level of socle degree $t$, we consider a
  trigrading on $S$ by setting $\deg(x_i)=(1,0,0)$ for $x_i \in X$,
  $\deg(x_i) = (0,1,0)$ for $x_i \in Y$, and $\deg(x_i)=(0,0,1)$ for
  $x_i \in Z$. We let $\xx=x_1\cdots x_j$, $\yy=x_{j+1} \cdots x_k$
  and $\zz=x_{k+1}\cdots x_m$. Given vectors $\aa,\bb$ and $\cc$ of
  appropriate length and with non-negative entries, we set
  $\xx^\aa, \yy^\bb$ and $\zz^\cc$ in analogy with the previous case.

  Observe that the minimal monomial generators of $J$ are all the
  monomials $\xx^\aa\yy^\bb\zz^\cc$ of degree $a+b+c=2u+1$ which
  satisfy one of the following four conditions:
  \begin{enumerate}
  \item $c=0$, $a$ is odd and $b$ is even.
  \item $a=0$, $b$ is odd and $c$ is even.
  \item $b=0$, $c$ is odd and $a$ is even.
  \item $abc \ne 0$, $a+b$ is even and $c$ is odd.
  \end{enumerate}
  We start by showing that $\m^{t+1} \subseteq J$. Assume that
  $\xx^\aa\yy^\bb\zz^\cc$ has degree $a+b+c \geq t+1 = 2(u+1)$; we
  distinguish some cases: if $abc \ne 0$, then we can find vectors
  $\aa' \leq \aa$, $\bb' \leq \bb$ and $\cc' \leq \cc$ with
  $|\aa'|+|\bb'|+|\cc'| = 2u-1$ and such that
  $\xx^\aa\yy^\bb\zz^\cc \in XYZ \cdot
  (\xx^{\aa'}\yy^{\bb'}\zz^{\cc'})$. We have already shown in the 
  case $t = 2(u-1)$ above, that $\m^{2u-1} \subseteq ((X+Y)^2+Z^2)^{u-1}$, and in
  particular we have that
  $\xx^{\aa'}\yy^{\bb'}\zz^{\cc'} \in ((X+Y)^2+Z^2)^{u-1}$. It follows
  that $\xx^\aa\yy^\bb\zz^\cc \in J$ in this case. Now assume that
  $abc=0$. If $a=0$, then since $b+c \geq 2(u+1)$ we can find
  $\bb' \leq \bb$ and $\cc' \leq \cc$ such that $|\bb'|$ is odd,
  $|\cc'|$ is even, and $|\bb'|+|\cc'| = 2u+1$. It follows that
  $\xx^\aa\yy^\bb\zz^\cc \in (\yy^{\bb'}\zz^{\cc'}) \subseteq
  Y(Y^2+Z^2)^u \subseteq J$. The remaining cases in which either $b=0$
  or $c=0$ are handled similarly. This concludes the proof that
  $\socdeg(A) \leq t$.

  To see that equality holds, as above it suffices to show that no
  monomial $w=\xx^\aa\yy^\bb\zz^\cc$ of degree $a+b+c=t-1=2u$ belongs
  to $J:\m$. Assume by way of contradiction that such a monomial $w$
  belongs to $J:\m$. If two of the exponents $a,b,c$ are equal to
  zero, for instance if $a=2u$ and $b=c=0$, then observe that
  $w \cdot x_{j+1} \in J$ must be a minimal monomial generator of $J$,
  by degree considerations. But
  $w \cdot x_{j+1} = \xx^\aa \cdot x_{j+1} = \xx^\aa\yy^{\bb'}\zz^{\bf
    0}$ with $\bb'=(1,0,\ldots,0)$ does not belong to $J$, thanks to
  the above characterization of the minimal generators of $J$. This
  contradicts our choice of $w \in J:\m$. The cases in which $a=b=0$
  or $a=c=0$ are handled similarly. Now assume that only one among
  $a,b,c$ is equal to zero. For instance, say $a=0$ and $b+c=2u$. If
  $b$ is even, then we consider
  $w\cdot x_{k+1} = \xx^{\bf 0}\yy^{\bb}\zz^{\cc'}$ with
  $\cc'=\cc+(1,0,\ldots,0)$, which must be a minimal generator of $J$,
  necessarily of type (2). However $b$ is even, again a
  contradiction. If $b$ is odd, then we consider
  $w \cdot x_{j+1} = \xx^{\bf 0}\yy^{\bb'}\zz^\cc$ with $|\bb'| = b+1$
  odd, and we reach again a similar contradiction. The other cases in
  which only one of the three exponents is equal to zero are tackled
  in a similar fashion.

  Finally, assume that $abc\ne 0$. If $c$ is odd, then
  $w \cdot x_{k+1} = \xx^{\aa}\yy^\bb\zz^{\cc'}$ with $|\cc'| = c+1$
  must be a minimal generator of $J$, necessarily of type (4); but
  $c+1$ is even, a contradiction. If $c$ is even, then so is
  $a+b = 2u-c$. Then $w \cdot x_1 = \xx^{\aa'}\yy^\bb\zz^\cc$ with
  $|\aa'| = a+1$ must be a minimal generator of $J$ of type (4);
  however $|\aa'| + |\bb| = a+1+b$ is odd, a contradiction.
\end{proof}

\begin{theorem}
  \label{generation}
  For all odd integers $s=2t-1 \geq 3$ and all $n\ge 4$ there exists
  an Artinian compressed Gorenstein algebra $A= K[x_1,\ldots,x_n]/I$
  with socle degree $s$ and such that $m(I)=t$.
\end{theorem}

\begin{proof}
  From Lemma~\ref{lemma:1} we can construct a monomial ideal
  $J \subseteq S= K[x_1,\dots,x_{n-1}]$ with $m(J)=t$ and $\soc(S/J)$
  concentrated only in degree $t$. From Hartshorne's construction
  \cite[Thm. 4.9]{Hartshorne}, there is a reduced set of points $X$ in
  $\mathbb P^{n-1}$ such that the Artinian reduction
  $K[x_1,\dots,x_n]/(I_X + (x_n))$ is isomorphic to $S/J$. We will now
  construct the desired Artinian Gorenstein algebra as a quotient of
  the homogenous coordinate ring  $R_X =K[x_1,\dots,x_n]/I_X$ by a
  \emph{doubling} \cite[\S 2.5]{KKRSSY}. There are many ways to describe this. One is to use
  the   canonical module $\omega_X$ which can be embedded as an ideal in
  $\a\subseteq R_X$ since $R_X$ is generically
  Gorenstein~\cite[Prop. 3.3.18 (a)]{BrunsHerzog}. This can be done with
  initial degree $i$ for all $i$ where $\dim_K((R_X)_i) = \deg(X)$ holds \cite[Corollary 1.7 and
  Proposition 1.9]{Kr00}. Since $t$ is the socle degree of the
  Artinian reduction $S/J$ of $R_X$ we have that $\reg(R_X)=t$, and hence  $\dim_K((R_X)_t) =\deg(X)$. 
  Let $\a$ be a canonical ideal with
  initial degree $t$. Since $S/J$ is level, $\omega_X$ is generated in
  a single degree. If we
  let $I \subseteq R=K[x_1,\ldots,x_n]$ be the ideal such that
  $I/I_X=\a$, then since both $\a \subseteq R_X$ and $I_X \subseteq R$
  are generated in degree $t$ we conclude that $I$ is generated in
  degree $t$ as well. Thus $A=R_X/\a \cong R/I$ is an Artinian
  Gorenstein $K$-algebra with defining ideal $I$ generated only in
  degree $t$ ~\cite[Prop. 3.3.18 (b)]{BrunsHerzog}. We now claim that $\socdeg(A) = 2t-1$. In fact, $\omega_X$ has
  regularity $\dim(R_X)=1$ and it is generated in degree $1-t=\dim(R_X)-\reg(R_X)$ \cite[13.4.7]{BrodmannSharp}. Since $\a$ is a canonical ideal generated in
  degree $t$ we therefore have that $\a \cong \omega_X(1-2t)$, and thus $\reg(\a)=2t$. In particular, $\dim_K(\a_i) = \deg(X)$ if and only if
  $i \geq 2t$. It follows from the graded short exact sequence $0 \to \a \to R_X \to A \to 0$ that $A_i=0$ for all $i \geq 2t$, while
  $A_{2t-1} \ne 0$. This shows the claim. Finally, since $I$ has no generators in degree
  less than $t$, we conclude that $\HF_A(i) = \binom{n+i-1}{i}$ for all $0 \leq i <t$. If we let $s=2t-1$, by symmetry of the Hilbert function of an Artinian Gorenstein $K$-algebra we get that $\HF_A(i) = \HF_A(s-i) = \binom{n+s-i-1}{s-i}$ for all $t \leq i \leq s$. Thus, $A$ is compressed.
\end{proof}

\begin{remark}
  The construction of Theorem \ref{generation} also works for $n=3$
  and $s$ an even integer. However, we point out that for $s =2t-1$
  with $t \geq 3$ odd there is no Artinian Gorenstein compressed
  algebra $K[x_1,x_2,x_3]/I$ generated only in degree $t$. In fact, if
  such an algebra existed, the ideal $I$ would have an even number of
  generators, contradicting the structure theorem of Buchsbaum and
  Eisenbud for ideals of codimension three defining Gorenstein rings
  \cite[Corollary 2.2]{BE}. For this reason we assume $n\ge 4.$
\end{remark}

As a consequence of the above results, we may conclude that Boij's
conjecture holds true for the first Betti number. 

\begin{corollary} \label{generic} Let $A$ be a generic Artinian Gorenstein
  $K$-algebra with $\edim(A) \geq 4$ and socle degree $s \geq 3$. Let $A=R/I$ be a minimal presentation. Then $I$ is generated in degree $\lfloor \frac{s}{2} \rfloor +1$.
\end{corollary}

\begin{proof}
  Since generic Artinian Gorenstein graded $K$-algebras are
  compressed by Theorem~\ref{compressed}, we conclude thanks to Theorem   \ref{generation} and Remark \ref{Boij Conj}. 
\end{proof}

\subsection{Rate of of generic Artinian Gorenstein algebras} { Thanks
  to Corollary \ref{generic} we have the following result.}

\begin{theorem} \label{main} Let $A$ be a generic Artinian Gorenstein
  graded $K$-algebra of socle degree $s \geq 3$, and assume that
  $\edim(A) \geq 4$. Then $\rate(A) = \lfloor \frac{s}{2} \rfloor $.
\end{theorem}

\begin{proof}
Let $A$ be a generic Artinian Gorenstein graded $K$-algebra, and let $A=R/I$ be a minimal presentation. If $s=3$, then $A$ is Koszul by \cite[Theorem 6.3]{CRV}, and thus $\rate(A) = 1 =\lfloor \frac{s}{2} \rfloor$. Now assume $s \geq 4$. By Theorem~\ref{compressed} $A$ is
  compressed, and hence it satisfies $\rate(A) = m(I)-1$ by Corollary \ref{corollary t-generated}. Finally, by Corollary \ref{generic},
  $m(I) = \lfloor \frac{s}{2} \rfloor+1$, and the result follows.
\end{proof}

\begin{remark} Recall that Artinian compressed Gorenstein $K$-algebras
  of even socle degree $s=2(t-1)$ have almost linear resolution and,
  in particular, they are generated in degree $t$.  Hence the
  conclusion of Theorem \ref{main} holds for any Artinian Gorenstein
  compressed $K$-algebra of even socle degree, not necessarily
  generic.
\end{remark}

\begin{remark} When $A=R/I$ is a generic Artinian Gorenstein
  $K$-algebra of odd socle degree $s=2t-1$ with $\edim(A)=3$ we have
  that
  \[
    \rate(A) = \begin{cases} t-1 & \text{ if } t \text{ is even}\\
      t & \text{ if } t \text{ is odd}
    \end{cases}.
  \]
  In fact, if $s=3$ then $A$ is Koszul by \cite[Theorem 6.3]{CRV}. If
  $s \geq 5$ then $I$ is generated in degree $t$ when $s=2t-1$ and $t$
  is even, while there is one generator in degree $t+1$ when $s=2t-1$
  and $t$ is odd. In both cases we conclude by Proposition \ref{prop
    compressed} since $A$ is compressed.
\end{remark}

\subsection*{Acknowledgments} We thank the anonymous referee for providing useful suggestions and comments. The second, third and fourth
  authors were partially supported by the PRIN 2020 project 2020355B8Y
  ``Squarefree Gröbner degenerations, special varieties and related
  topics'', by the MIUR Excellence Department Project CUP D33C23001110001, and by INdAM-GNSAGA. 

\bibliographystyle{alpha} \bibliography{References}
\end{document}